\documentclass[12pt,leqno,oneside]{amsart}
\usepackage{mathrsfs,dsfont}
\usepackage{amsmath,amstext,amsthm,amssymb,bbm}
\usepackage{charter}
\usepackage{typearea}
\usepackage{pdfsync}
\usepackage{hyperref}
\usepackage{color}

\usepackage[width=6.4in,height=8.5in]
{geometry}

\numberwithin{equation}{section}

\newtheorem{Theorem}{Theorem}[section]
\newtheorem{Proposition}[Theorem]{Proposition}
\newtheorem{Lemma}[Theorem]{Lemma}

\theoremstyle{definition}

\newtheorem{Example}[Theorem]{Example}

\newcommand{\db}{\overline\partial}
\newcommand{\ov}{\overline}
\newcommand{\wi}{\widetilde}

\newcommand{\FS}{{\rm FS}}

\DeclareMathOperator{\ric}{Ric}

\DeclareMathOperator{\supp}{supp}

\DeclareMathOperator{\ord}{ord}
\DeclareMathOperator{\di}{div}

\newcommand{\cali}[1]{\mathscr{#1}}
\newcommand{\cO}{\cali{O}}

\newcommand{\cC}{\cali{C}}

\newcommand{\field}[1]{\mathbb{#1}}
\newcommand{\Z}{\field{Z}}
\newcommand{\R}{\field{R}}
\newcommand{\C}{\field{C}}
\newcommand{\N}{\field{N}}

\newcommand{\D}{\field{D}}
\newcommand{\E}{\field{E}}

\newcommand{\comment}[1]{}

\begin{document}

\title{Equidistribution for weakly holomorphic sections of line bundles on algebraic curves}

\author{Dan Coman}
\thanks{D.\ Coman is partially supported by the NSF Grant DMS-1700011}
\address{Department of Mathematics, 
Syracuse University, Syracuse, NY 13244-1150, USA}
\email{dcoman@syr.edu}

\author{George Marinescu}
\address{Univerisit\"at zu K\"oln, Mathematisches institut,
Weyertal 86-90, 50931 K\"oln, Germany 
\newline\mbox{\quad}\,Institute of Mathematics `Simion Stoilow', 
Romanian Academy, Bucharest, Romania}
\email{gmarines@math.uni-koeln.de}
\thanks{G.\ M.\ partially supported by 
the DFG funded project SPP 2265
(Project-ID 422743078)}

\subjclass[2010]{Primary 32L10; Secondary 14H60, 30F10, 32U40.}
\keywords{Bergman kernel, Fubini-Study current, 
singular Hermitian metric, algebraic curve, weakly holomorphic sections}

\date{December 28, 2020}
\dedicatory{Dedicated to Professor Ahmed Zeriahi on the occasion of his retirement}

\begin{abstract}
We prove the convergence of the normalized Fubini-Study measures and the logarithms of the Bergman kernels of various Bergman spaces of holomorphic and weakly holomorphic sections associated to a singular Hermitian holomorphic line bundle on an algebraic curve. Using this, we study the asymptotic distribution of the zeros of random sequences of sections in these spaces.
\end{abstract}

\maketitle

\tableofcontents

\section{Introduction}

Let $(X,\omega)$ be a K\"ahler manifold of dimension $n$ and $(L,h)$ be a positive holomorphic line bundle on $X$ such that $\omega=c_1(L,h)$. We let $h_p$ be the metric induced by $h$ on $L^p:=L^{\otimes p}$ and denote by $H^0(X,L^p)$ the space of holomorphic sections of $L^p$. One can define a sequence of Fubini-Study forms $\gamma_p$ on $X$ by setting $\gamma_p=\Phi^\star(\omega_\FS)$, where $\Phi:X\to\mathbb P\big(H^0(X,L^p)^\star\big)$ is the Kodaira map associated to $(L^p,h_p)$ and $\omega_\FS$ is the Fubini-Study form on a projective space. A theorem of Tian \cite{Ti90} states that $\frac{1}{p}\,\gamma_p\to c_1(L,h)$ as $p\to+\infty$, in the $\cC^2$ topology on $X$. In \cite{CM15} we proved the analogue of Tian's theorem in the case when $h$ is a singular metric on $L$ and its curvature is a K\"ahler current, i.e.\ $c_1(L,h)\geq\varepsilon\omega$ for some $\varepsilon>0$. The above convergence is now in the weak sense of currents on $X$. 

In \cite{CMM17} we generalized this further and studied 
the asymptotic behavior of the sequence of Fubini-Study currents 
associated to an arbitrary sequence of singular Hermitian 
holomorphic line bundles $L_p$ on a compact normal K\"ahler space $X$. 
The normality of $X$ was essential, in order to apply 
Riemann's second extension theorem for holomorphic 
functions \cite[page 143]{GR84} and for plurisubharmonic functions 
\cite[Satz 4]{GR56} on a normal complex space. 
An interesting question is to analyze the general case when $X$ 
is a compact K\"ahler space not necessarily normal.

In the present paper we study the one dimensional case.
Note that any compact one dimensional complex space
is projective and thus algebraic by \cite[Satz\,2,\,p.\,343]{Gra:62} 
(see also \cite[Theorem 6.2]{RS15}).
We consider the following setting:

\smallskip

(A) $X\subset{\mathbb P}^N$ is an irreducible algebraic curve, 
$\Sigma=\{x_1,\ldots,x_m\}\subset X$ is the set of singular points 
of $X$, and $\omega$ is a Hermitian form on $X$. 

\smallskip

(B) $L$ is a holomorphic line bundle on $X$ 
with singular Hermitian metric  $h$ whose local weights 
are weakly subharmonic and such that 
$c_1(L,h)\geq \varepsilon\omega$ on $X\setminus\Sigma$ 
for some $\varepsilon>0$.

\smallskip

Let us now introduce the normalization of $X$, 
which will be needed throughout the paper:

\smallskip

(C) $\sigma:\wi X\to{\mathbb P}^N$, where $\wi X$ is a 
compact Riemann surface, is the normalization of $X$ and 
$\wi\omega$ is a Hermitian form on $\wi X$.

\smallskip

Here $c_1(L,h)$ denotes the curvature measure of $h$ 
(see Section \ref{S:BFS}). A natural choice of the form $\omega$ 
is the restriction to $X$ of the Fubini-Study form on ${\mathbb P}^N$. 
We denote by $H^0_w(X,L^p)$, respectively by $H^0_c(X,L^p)$, 
the space of weakly holomorphic sections, 
respectively continuous weakly holomorphic sections of $L^p$. 
Then 
\[H^0(X,L^p)\subset H^0_c(X,L^p)\subset 
H^0_w(X,L^p)\subset H^0(X\setminus\Sigma,L^p),\]
where the latter is the space of holomorphic sections of 
$L^p\,\vert_{X\setminus\Sigma}$. 
We refer to Section \ref{S:prelim} and Section \ref{S:BFS} 
for the definitions of these notions.

We consider the corresponding Bergman subspaces of 
$L^2$-holomorphic sections with respect to the natural inner product 
induced by the metric $h_p$ and $\omega$:
\begin{align}
H^0_{c,(2)}(X,L^p)&=
\big\{S\in H^0_c(X,L^p):\,\|S\|_p<+\infty\big\}, \label{e:Bsc}\\
H^0_{w,(2)}(X,L^p)&=
\big\{S\in H^0_w(X,L^p):\,\|S\|_p<+\infty\big\},  \label{e:Bsw}\\
H^0_{(2)}(X\setminus\Sigma,L^p)&=
\big\{S\in H^0(X\setminus\Sigma,L^p):\,\|S\|_p<+\infty\big\}, 
\label{e:Bsreg}
\end{align}
where 
\[\|S\|^2_p=\int_{X\setminus\Sigma}|S|^2_{h_p}\,\omega.\]
We show in Proposition \ref{P:Bsreg} that the spaces 
$H^0_{(2)}(X\setminus\Sigma,L^p)$ are finite dimensional.

Let $P_{c,p},P_{w,p},P_p$ be the Bergman kernel functions 
and $\gamma_{c,p},\gamma_{w,p},\gamma_p$ 
be the Fubini-Study measures of the spaces 
$H^0_{c,(2)}(X,L^p)$, $H^0_{w,(2)}(X,L^p)$, 
$H^0_{(2)}(X\setminus\Sigma,L^p)$, respectively, 
induced by the above metric data. 
We refer to Section \ref{S:BFS} for their definition and properties. 
In particular, we have that $\log P_{c,p},\log P_{w,p},\log P_p\in L^1(X,\omega)$ 
(see \eqref{e:BFS2}, \eqref{e:BFS5}), and $\gamma_p$ 
are signed measures with ``small" negative variation 
(see Lemma \ref{L:FSneg}). Our main result is the following:

\begin{Theorem}\label{T:mt}
Let $X,\omega,L,h$ verify assumptions (A) and (B). Then, as $p\to+\infty$, we have: 

(i) $\displaystyle\frac{1}{p}\,\log P_{c,p}\to0,\,
\frac{1}{p}\,\log P_{w,p}\to0,\,\frac{1}{p}\,
\log P_p\to0$ in $L^1(X,\omega)$.

(ii) $\displaystyle\frac{1}{p}\,\gamma_{c,p}\to c_1(L,h),\,
\frac{1}{p}\,\gamma_{w,p}\to c_1(L,h),\,\frac{1}{p}\,
\gamma_p\to c_1(L,h),\,\frac{1}{p}\,\gamma_p^+\to c_1(L,h)$ 
in the weak sense of measures on $X$, where $\gamma_p^+$ 
is the positive variation of $\gamma_p$.
\end{Theorem}

In a series of papers starting with \cite{ShZ99}, 
Shiffman and Zelditch describe the asymptotic distribution of zeros 
of random sequences of holomorphic sections of powers of a 
positive line bundle $L$ on a projective manifold 
(see also \cite{ShZ08}, \cite{Sh08}, \cite{DMS}). 
We showed in \cite{CM15} that some of their results 
can be generalized to the setting of line bundles $L$ 
with singular Hermitian metrics on compact K\"ahler manifolds, 
and also on compact K\"ahler orbifolds \cite{CM13}. 
Further such equidistribution results with estimates on 
the speed of convergence are obtained in 
\cite{CMN16}, \cite{CMN18}, \cite{DMM}.

In \cite{CMM17} and \cite{BCM20} we prove equidistribution results for zeros of random sequences of sections in the case when the base space $X$ is a compact normal K\"ahler space and the sequence of powers $L^p$ is replaced by an arbitrary sequence of singular Hermitian holomorphic line bundles $L_p$ satisfying certain assumptions. The results in \cite{BCM20} apply to very general probability measures on the Bergman spaces $H^0_{(2)}(X,L_p)$ of $L^2$-integrable holomorphic sections, and one of the key ingredients of the proof is the version of Theorem \ref{T:mt} in that setting (see \cite[Theorem 1.1]{CMM17} and \cite[Theorem 1.1]{BCM20}). 

We conclude the paper by noting that  \cite[Theorem 1.1]{BCM20} holds in the present setting without any changes and with the same proof. Let $\{H^p\}_{p\geq1}$ be any of the sequences of Bergman spaces defined in \eqref{e:Bsc}, \eqref{e:Bsw}, or \eqref{e:Bsreg}. Given a section $S\in H^p$, we associate to $S$ the measure $[\di(S)]$ defined in \eqref{e:div1}. Geometrically, $[\di(S)]$ is the sum of Dirac masses with multiplicities at the zeros of $S$ in $X\setminus\Sigma$ plus a sum of Dirac masses at the singular points $x_j\in\Sigma$ with coefficients given in terms of the order of $\sigma^\star S$ at the points of $\sigma^{-1}(x_j)$ (see Proposition \ref{P:Bsreg} and \eqref{e:div2}). 

We set $n_p=\dim H^p$ and let $S^p_1,\ldots,S^p_{n_p}$ be an orthonormal basis of $H^p$. Using this basis we identify $H^p$ to $\C^{n_p}$ and endow it with a probability measure $\mu_p$ which satisfies the moment condition (B) from \cite{BCM20}. Then the conclusions of \cite[Theorem 1.1]{BCM20} hold with $A_p=p$ in the setting of Theorem \ref{T:mt} for the probability spaces $(H^p,\mu_p)$. Let us state here more precisely one particular case of this theorem.

We let $\mu_p$ be the normalized area measure 
on the unit sphere of $H^p\equiv\C^{n_p}$ 
(see \cite[(4.13)]{BCM20}) and consider 
the product probability space 
$(\mathcal{H},\mu)=\left(\prod_{p=1}^\infty 
H^p,\prod_{p=1}^\infty\mu_p\right)$. 
The expectation measure $\E[\di(s_p)]$ 
of the measure-valued random variable 
$H^p\ni s_p\mapsto[\di(s_p)]$ is defined by
\[\big\langle \E[\di(s_p)],\chi\big\rangle=
\int_{H^p}\Big(\int_X\chi\,d[\di(s_p)]\Big)\,d\mu_p(s_p),\]
where $\chi$ is a continuous function on $X$. We have:

\begin{Theorem}\label{T:zero}
Let $X,\omega,L,h$ verify assumptions (A) and (B). Then the following hold:

(i) The measure $\E[\di(s_p)]$ is well defined and $\displaystyle\frac{1}{p}\,\E[\di(s_p)]\to c_1(L,h)$\,, as $p\to+\infty$, in the weak sense of measures on $X$. 

(ii) For $\mu$-a.\,e.\ sequence $\{s_p\}\in\mathcal{H}$ we have $\displaystyle\frac{1}{p}\,\log|s_p|_{h_p}\to0$ in $L^1(X,\omega)$ and $\frac{1}{p}\,[\di(s_p)]\to c_1(L,h)$, $\frac{1}{p}\,[\di(s_p)]^+\to c_1(L,h)$, in the weak sense of measures on $X$, as $p\to+\infty$, where $[\di(s_p)]^+$ is the positive variation of $[\di(s_p)]$.
\end{Theorem}

The paper is organized as follows. In Section \ref{S:prelim} 
we recall the notion of weakly subharmonic function on a complex 
curve and the definition of its Laplacian. In Section \ref{S:BFS} 
we consider holomorphic line bundles endowed with singular metrics 
on an algebraic curve. We discuss the measures associated to divisors 
of holomorphic or weakly holomorphic sections and 
we define the Bergman kernel functions and Fubini-Study measures 
of the corresponding Bergman spaces. Theorems \ref{T:mt} and \ref{T:zero} 
are proved in Section \ref{S:proofs}. In Section \ref{S:ex} we give examples 
of algebraic curves in ${\mathbb P}^2$ for which we describe explicitly 
the Bergman spaces of sections considered in the paper. 
We also give a precise lower estimate
of the Bergman kernel $P_{w,p}$
in the case of a smooth Hermitian metric on $L$.

\section{Preliminaries}\label{S:prelim}

In this section we review the notions of (weakly) 
holomorphic and (weakly) subharmonic function on a complex curve. 

Throughout the paper we denote by $\D_r\subset\C$ 
the open disc of radius $r>0$ centered at $0$, by $\D:=\D_1$ 
the unit disc, and by $\lambda$ the Lesbesgue measure on $\C$. 
We denote by $\ord(g,\zeta)$ the order of a meromorphic function 
$g$ at $\zeta\in\C$. 

Let $X$ be a complex curve, i.e.\ a reduced complex space 
of dimension one, and let $\omega$ be a Hermitian form on $X$ 
(see e.g.\ \cite[Section 2.1]{CMM17} for the definition). 
Let $\Sigma$ be the set of singular points of $X$. 

Working locally near a singular point $x_j\in\Sigma$ and 
using a local embedding of $X$ into $\C^N$ for some $N\geq2$, 
we may assume that $X$ is a complex curve in a polydisc 
$D_{x_j}\subset\C^N$ centered at $x_j$, and is the union 
of finitely many irreducible complex curves which intersect only at $x_j$. 
Moreover each such irreducible component $Y$ of $X$ at $x_j$ 
has a local normalization (see \cite[Theorem 5.7]{Gri89}, 
\cite[Section 6.1]{Ch89}): 
\begin{equation}\label{e:ln}
\begin{split}
& f=(f_1,\ldots,f_N):\D\to D_{x_j}  \text{ holomorphic with $f(\D)=Y$, 
$f(0)=x_j$, and} \\
& \text{$f:\D\setminus\{0\}\to Y\setminus\{x_j\}$ is biholomorphic.}
\end{split}
\end{equation}
We denote by 
\begin{equation}\label{e:ram1}
\alpha=\alpha(x_j,Y)=\min\{\ord(f'_\ell,0):\,1\leq\ell\leq N\}
\end{equation}
the ramification index of $f$ at $0$ (see e.g.\ \cite[p.\ 264]{GH94}). 

By shrinking the polydisc $D_{x_j}$, we may assume that 
$\omega$ is the restriction to $X$ of a Hermitian form on $D_{x_j}$. Hence 
\begin{equation}\label{e:beta}
C_1^{-1}\beta_N\,\vert_{_X}\leq\omega\leq C_1\beta_N\,\vert_{_X}
\end{equation}
for some constant $C_1>1$, where 
$\beta_N=\frac{i}{2}\,\sum_{\ell=1}^Ndz_\ell\wedge d\ov z_\ell$ 
is the standard K\"ahler form on $\C^N$. By \eqref{e:beta} 
we have $C_1^{-1}f^\star\beta_N\leq f^\star\omega\leq C_1f^\star\beta_N$. 
Hence we may assume that 
\begin{equation}\label{e:ram2}
C_2^{-1}|\zeta|^{2\alpha}\,\frac{i}{2}\,d\zeta\wedge d\ov\zeta\leq 
f^\star\omega\leq C_2|\zeta|^{2\alpha}\,\frac{i}{2}\,d\zeta\wedge d\ov\zeta\leq 
C_2\,\frac{i}{2}\,d\zeta\wedge d\ov\zeta
\end{equation}
holds on $\D$, with some constant $C_2>1$. 

If $U\subset X$ is an open set, a weakly holomorphic function 
on $U$ is a holomorphic function on $U\setminus\Sigma$ 
which is locally bounded on $U$. A weakly holomorphic function 
on $U$ that extends continuously at the points of $\Sigma\cap U$ 
is called a continuous weakly holomorphic function on $U$. 
Note that such a function is not necessarily holomorphic on $U$ 
(see e.g.\ \cite[p.\ 91]{Gu2}). Let 
$$\cO_X(U)\subset\cO_{X,c}(U)\subset\cO_{X,w}(U)$$ 
denote the set of holomorphic, continuous weakly holomorphic, 
respectively weakly holomorphic functions on $U$. 
We remark that if $X$ is locally irreducible at any point
then the sheaves $\cO_{X,c}$ and $\cO_{X,w}$
coincide.

A subharmonic function on $U$ is a function which 
(using local embeddings $X\hookrightarrow\C^N$) 
is locally the restriction to $X$ of a plurisubharmonic function 
in the ambient space $\C^N$ (see e.g.\ \cite[Section 2.1]{CMM17}). 
A weakly subharmonic function on $U$ is a subharmonic function 
on $U\setminus\Sigma$ which is locally upper bounded on $U$. 
Let $SH(U)\subset WSH(U)$ be the set of subharmonic, 
respectively weakly subharmonic functions on $U$.

\begin{Lemma}\label{L:wsh1}
If $U\subset X$ is open then $WSH(U)\subset L^1_{loc}(U,\omega)$.
\end{Lemma}

\begin{proof} 
Let $u\in WSH(U)$. Since $u$ is subharmonic on 
$U\setminus\Sigma$ we have that 
$u\in L^1_{loc}(U\setminus\Sigma,\omega)$. 
So we only need to show that $u$ is integrable on each 
irreducible component $Y$ of $X$ at a point $x_j\in\Sigma\cap U$. 
Let $f:\D\to Y$ be a local normalization of $Y$ and set 
$Y_r=f^{-1}(\D_r)$ for $r<1$. Then $u\circ f$ is subharmonic on 
$\D\setminus\{0\}$ and upper bounded near $0$, 
so it extends to a subharmonic function on $\D$. 
Hence by \eqref{e:ram2},
\[\int_{Y_r\setminus\{x_j\}}|u|\,\omega=
\int_{\D_r\setminus\{0\}}|u\circ f|\,f^\star\omega\leq 
C_2\int_{\D_r}|u\circ f|\,d\lambda<+\infty.\]
This yields the conclusion.
\end{proof}
We refer to \cite{D85} (see also \cite[Section 2.1]{CMM17}) 
for the definition of smooth forms on complex spaces. 
In our context let $\cC_{X,0}^\infty(U)$ denote 
the set of smooth functions on $X$ with compact support in $U$. 
Let $d=\partial+\db$, $d^c=\frac{1}{2\pi i}\,(\partial-\db)$, so 
$dd^c=\frac{i}{\pi}\,\partial\db$. If $u\in WSH(U)$ then $dd^cu$ 
is a positive measure on $U\setminus\Sigma$ given in a local coordinate by 
$\frac{1}{2\pi}\,\Delta u$. Since $u\in L^1_{loc}(U,\omega)$ 
we have that $dd^cu$ is a distribution on $U$ defined by 
\[\langle dd^cu,\chi\rangle=
\int_{U\setminus\Sigma}u\,dd^c\chi\,,\,\;\chi\in\cC_{X,0}^\infty(U).\]

\begin{Lemma}\label{L:wsh2}
If $u\in WSH(U)$ then $dd^cu$ is a positive measure on $U$. 
Moreover, assume that $x_j\in\Sigma\cap U$, $D_{x_j}\subset\C^N$ 
is a polydisc as in \eqref{e:ln}, and $Y_\ell$ are the irreducible components 
of $X$ at $x_j$ with local normalizations $f_\ell:\D\to Y_\ell$, $1\leq\ell\leq k$. 
Then each function $v_\ell:=u\circ f_\ell$ extends 
to a subharmonic function on 
$\D$ and $dd^cu=\sum_{\ell=1}^k(f_\ell)_\star(dd^cv_\ell)$. 
\end{Lemma}

\begin{proof} 
Since $v_\ell$ is subharmonic on $\D\setminus\{0\}$ 
and is upper bounded near $0$, it extends to a subharmonic 
function on $\D$. If $\chi$ is a test function supported in $D_{x_j}$ we have 
\begin{align*}
\langle dd^cu,\chi\rangle&=
\sum_{\ell=1}^k\int_{Y_\ell\setminus\{x_j\}}u\,dd^c\chi
=\sum_{\ell=1}^k\int_{\D_\ell\setminus\{0\}}(u\circ f_\ell)\,
dd^c(\chi\circ f_\ell) \\ 
&=\sum_{\ell=1}^k\langle dd^cv_\ell,\chi\circ f_\ell\rangle=
\sum_{\ell=1}^k\langle(f_\ell)_\star(dd^cv_\ell),\chi\rangle.
\end{align*}
This yields the conclusion.
\end{proof}

We conclude this section with the following lemma 
(see also \cite[Theorem 1.7]{D85}):

\begin{Lemma}\label{L:wsh3}
Let $u:U\to[-\infty,+\infty)$ be a function. 
Then $u\in SH(U)$ if and only if $u\in WSH(U)$ and, 
for every $x_j\in\Sigma\cap U$ and every irreducible component 
$Y$ of $X$ at $x_j$, we have that 
$\displaystyle u(x_j)=\limsup_{Y\ni x\to x_j}u(x)$.
\end{Lemma}

\begin{proof} 
One implication is obvious, so we assume that $u\in WSH(U)$. 
Using \cite[Theorem 5.3.1]{FN80} we have to show that $u\circ g$ 
is subharmonic on $\D$, for any non-constant 
holomorphic function $g:\D\to U$. It suffices to assume that 
$g(0)=x_j\in\Sigma\cap U$ and to prove that $u\circ g$ 
is subharmonic on $\D_\varepsilon$ for some $\varepsilon>0$. 
For $\varepsilon>0$ sufficiently small we have that $Y=g(\D_\varepsilon)$ 
is an irreducible component of $X$ at $x_j$ such that $Y\setminus\{x_j\}$ 
is smooth and $g(\D_\varepsilon\setminus\{0\})=Y\setminus\{x_j\}$. 
The function $v=u\circ g$ is subharmonic on $\D_\varepsilon\setminus\{0\}$ and 
\[\limsup_{\zeta\to0}v(\zeta)=\limsup_{Y\ni x\to x_j}u(x)=u(x_j)=v(0).\]
Hence $v$ is subharmonic on $\D_\varepsilon$.
\end{proof}
We note that if $X$ is locally irreducible at any point
then the notions of subharmonic and weakly subharmonic function are the same.

\section{Bergman kernels and Fubini-Study measures}\label{S:BFS} 

We assume in this section that $X$, $\Sigma$, $\omega$, 
$\sigma:\wi X\to{\mathbb P}^N$, $\wi\omega$, verify (A) and (C). 
We introduce and study the Bergman kernels and Fubini-Study measures 
for the various spaces of $L^2$-integrable holomorphic sections 
considered in this paper. 

By the properties of the normalization (see e.g.\ \cite{Gri89}) 
we have that 
$\sigma(\wi X)=X$, 
$$\sigma:\wi X\setminus\sigma^{-1}(\Sigma)\to X\setminus\Sigma$$ 
is biholomorphic, and the number of points in the preimage 
$\sigma^{-1}(x_j)$ of $x_j\in\Sigma$ is equal to the number 
of irreducible components of $X$ at $x_j$. 
Moreover, for each such component $Y$ there exists a unique point 
$y\in\sigma^{-1}(x_j)$ and a coordinate neighborhood of 
$y\equiv0$ which contains $\D$, such that $\sigma\,\vert_\D$ 
is a local normalization of $Y$ as in \eqref{e:ln} with the polydisc 
$D_{x_j}\subset\C^N\hookrightarrow{\mathbb P}^N$. 
We denote by $\alpha(y)=\alpha(x_j,Y)$ the ramification index of 
$\sigma$ at $y$ defined in \eqref{e:ram1}. Let 
\begin{equation}\label{e:ramdiv}
R_\sigma=\sum_{y\in\sigma^{-1}(\Sigma)}\alpha(y)y\,,\,\;
[R_\sigma]=\sum_{y\in\sigma^{-1}(\Sigma)}\alpha(y)\delta_y\,,
\end{equation}
be the ramification divisor of $\sigma$ and its associated measure, 
where $\delta_y$ is the Dirac mass at $y$. Let $\cO_{\wi X}(R_\sigma)$ 
denote the line bundle defined by $R_\sigma$.

We state the following version of Lemma \ref{L:wsh2} 
in the compact setting:

\begin{Lemma}\label{L:wsh4} 
If $u\in WSH(U)$, where $U\subset X$ is open, then 
$v=u\circ\sigma$ extends to a subharmonic function on $\sigma^{-1}(U)$ and $dd^cu=\sigma_\star(dd^cv)$. 
\end{Lemma}

\begin{proof}
Note that $v\in SH(\sigma^{-1}(U\setminus\Sigma))$ 
and $v$ is upper bounded near each point of $\sigma^{-1}(\Sigma\cap U)$, 
so it extends to a subharmonic function on $\sigma^{-1}(U)$. 
If $\chi\in\cC_{X,0}^\infty(U)$ we have 
\[\langle dd^cu,\chi\rangle=\int_{U\setminus\Sigma}u\,dd^c\chi=
\int_{\sigma^{-1}(U)}v\,dd^c(\chi\circ\sigma)=
\langle dd^cv,\chi\circ\sigma\rangle.\]
\end{proof}

Let $L\to X$ be a holomorphic line bundle and $\{U_\alpha\}$ 
be an open cover of $X$ such that $L$ has a holomorphic frame 
$e_\alpha$ on $U_\alpha$. 
We define $H^0_w(X,L)$, respectively $H^0_c(X,L)$, 
by requiring that $S\in H^0_w(X,L)$, respectively $S\in H^0_c(X,L)$,
if and only if $S\in H^0(X\setminus\Sigma,L)$
and for any $\alpha$ we have $s_\alpha\in\cO_{X,w}(U_\alpha)$, 
respectively $s_\alpha\in\cO_{X,c}(U_\alpha)$,
where $S=s_\alpha e_\alpha$ on $U_\alpha\setminus\Sigma$.  

The notion of singular Hermitian metric $h$ on $L$ 
is defined exactly as in the smooth case 
(see \cite{D90}, \cite[p.\,97]{MM07}, \cite[Section 2.2]{CMM17})). 
We have $|e_\alpha|^2_h=e^{-2\varphi_\alpha}$, 
where $\varphi_\alpha\in L^1_{loc}(U_\alpha,\omega)$ 
are called the local weights of $h$. We assume in the sequel that 
$h$ has weakly subharmonic weights, 
i.e.\ $\varphi_\alpha\in WSH(U_\alpha)$. 
If $g_{\alpha\beta}=e_\beta/e_\alpha\in\cO^*_X(U_\alpha\cap U_\beta)$ 
are the transition functions of $L$ then 
$\varphi_\alpha=\varphi_\beta+\log|g_{\alpha\beta}|$ 
holds on $(U_\alpha\cap U_\beta)\setminus\Sigma$. Set 
\begin{equation}\label{e:Chern}
c_1(L,h)\,\vert_{ U_\alpha}=dd^c\varphi_\alpha.
\end{equation}
It follows from Lemma \ref{L:wsh2} that $c_1(L,h)$ 
is a well defined positive measure on $X$, 
called the curvature measure of $h$. 

Let $\sigma^\star L\to\wi X$ be the pullback of the line bundle $L$, 
endowed with the pullback metric $\sigma^\star h$. 
Since $\sigma^\star h$ has weight 
$\varphi_\alpha\circ\sigma$ on $\sigma^{-1}(U_\alpha)$, 
we infer by Lemma \ref{L:wsh4} that $\sigma^\star h$ 
has subharmonic weights and 
\[
c_1(L,h)=
\sigma_\star\big(c_1(\sigma^\star L,\sigma^\star h)\big).
\]

Using the Riemann removable singularity theorem, 
it follows easily that the map 
\begin{equation}\label{e:iso}
\sigma^\star:H^0_w(X,L^p)\to H^0\big(\wi X,\sigma^\star L^p\big)
\end{equation}
is well defined and an isomorphism. Hence by the Riemann-Roch theorem
\cite[16.9]{For81} or Siegel's lemma 
(see \cite[Lemma 2.2.6]{MM07}), there exists a constant $C>0$ such that 
\[\dim H^0_w(X,L^p)\leq Cp\, \text{ for all $p\geq1$.}\] 

\medskip

We show next that the Bergman spaces 
$H^0_{(2)}(X\setminus\Sigma,L^p)$ defined in \eqref{e:Bsreg} 
are finite dimensional, as they correspond to spaces of 
meromorphic section of $\sigma^\star L^p$ with poles in $R_\sigma$. 
We need the following simple lemma:

\begin{Lemma}\label{L:Ls}
Let $g(\zeta)=\sum_{j=-\infty}^{+\infty}a_j\zeta^j$ 
be a holomorphic function on $\D\setminus\{0\}$ such that 
$\int_{\D\setminus\{0\}}|g(\zeta)|^2|\zeta|^{2n}\,d\lambda<+\infty$ 
for some $n\in\Z$. Then $a_j=0$ for all $j\leq-n-1$.
\end{Lemma}

\begin{proof}
Let $\varepsilon\in(0,1)$. Using polar coordinates we obtain 
\[\int_{\{\varepsilon<|\zeta|<1\}}|g(\zeta)|^2|\zeta|^{2n}\,
d\lambda=2\pi\sum_{j=-\infty}^{+\infty}|a_j|^2
\int_\varepsilon^1r^{2j+2n+1}\,dr.\]
The conclusion follows by letting $\varepsilon\to0$.
\end{proof}

\begin{Proposition}\label{P:Bsreg}
The map $\sigma^\star:H^0_{(2)}(X\setminus\Sigma,L^p)\to 
H^0\big(\wi X,\sigma^\star L^p\otimes\cO_{\wi X}(R_\sigma)\big)$ 
is well defined and injective. We have 
$\dim H^0_{(2)}(X\setminus\Sigma,L^p)\leq Cp$ 
for all $p\geq1$, where $C>0$ is a constant.
\end{Proposition}

\begin{proof}
The holomorphic sections of 
$\sigma^\star L^p\otimes\cO_{\wi X}(R_\sigma)$ 
can be identified to meromorphic sections of 
$\sigma^\star L^p$ with poles in $R_\sigma$, 
so we have to show that if $S\in H^0_{(2)}(X\setminus\Sigma,L^p)$ 
and $y\in\sigma^{-1}(\Sigma)$ then $\sigma^\star S$ 
has a pole of order at most $\alpha(y)$ at the isolated singularity $y$. 

Let $x_j=\sigma(y)$ and $\D$ be the unit disc in a 
coordinate neighborhood of $y\equiv0$ such that $\sigma\,\vert_\D$ 
is the local normalization of an irreducible component $Y$ of $X$ at $x_j$. 
We may assume that $x_j$ has a neighborhood $U_\alpha$ 
on which $L$ has a local holomorphic frame $e_\alpha$ 
such that $U_\alpha\cap\Sigma=
\{x_j\}$ and $\D\subset\sigma^{-1}(U_\alpha)$. 
Then $S=s_\alpha e_\alpha^{\otimes p}$ for some 
$s_\alpha\in\cO_X(U_\alpha\setminus\{x_j\})$, and 
$|e_\alpha|_h=e^{-\varphi_\alpha}$, 
where $\varphi_\alpha\in WSH(U_\alpha)$. 
We may assume that $\varphi_\alpha\leq M$ on $U_\alpha$, 
for some constant $M$. Using \eqref{e:ram2} we obtain 
\begin{align*}
\|S\|_p^2 & \geq\int_{Y\setminus\{x_j\}}|s_\alpha|^2
e^{-2p\varphi_\alpha}\,\omega
\geq e^{-2pM}\int_{\D\setminus\{0\}}|s_\alpha\circ\sigma|^2\,
\sigma^\star\omega \\
& \geq C_2^{-1}e^{-2pM}
\int_{\D\setminus\{0\}}|s_\alpha(\sigma(\zeta))|^2
|\zeta|^{2\alpha(y)}\,d\lambda.
\end{align*}
By Lemma \ref{L:Ls} we infer that the function 
$s_\alpha\circ\sigma$ has a pole of order at most $\alpha(y)$ at $y$.

The map $\sigma^\star$ is clearly injective, since 
$\sigma:\wi X\setminus\sigma^{-1}(\Sigma)\to X\setminus\Sigma$ 
is biholomorphic. 
The last assertion follows from the Riemann-Roch theorem
\cite[16.9]{For81} or
Siegel's lemma (see \cite[Lemma 2.2.1]{MM07} and its proof).
\end{proof}

\medskip

Let $S\in H^0_{(2)}(X\setminus\Sigma,L)$, $S\neq0$. 
It follows from Proposition \ref{P:Bsreg} that $\sigma^\star S$ 
is a meromorphic section of $\sigma^\star L$, so it induces the divisor 
$\di(\sigma^\star S)$ and its associated signed measure 
$[\di(\sigma^\star S)]$ on $\wi X$, where 
\[\di(\sigma^\star S)=\sum_{y\in\wi X}
\ord(\sigma^\star S,y)y\,,\,\;[\di(\sigma^\star S)]=
\sum_{y\in\wi X}\ord(\sigma^\star S,y)\delta_y.\]
Moreover, the divisor $\di(\sigma^\star S)+R_\sigma$ is effective. 
We define
\begin{equation}\label{e:ord}
\ord(S,x_j)=\sum_{y\in\sigma^{-1}(x_j)}\ord(\sigma^\star S,y),\;x_j\in\Sigma.
\end{equation}

Writing $S=s_\alpha e_\alpha$, where 
$s_\alpha\in\cO_X(U_\alpha\setminus\Sigma)$, 
we have that $s_\alpha\circ\sigma$ is meromorphic on 
$\sigma^{-1}(U_\alpha)$, hence $\log|s_\alpha\circ\sigma|$ 
is locally the difference of two subharmonic functions. 
We infer that $\log|s_\alpha|$ is locally the difference 
of two weakly subharmonic functions on $U_\alpha$. 
By Lemma \ref{L:wsh1}, this implies that 
$\log|s_\alpha|\in L^1_{loc}(U_\alpha,\omega)$. 
Since $\log|s_\alpha|=\log|s_\beta|+\log|g_{\alpha\beta}|$ 
on $(U_\alpha\cap U_\beta)\setminus\Sigma$, 
we can define, using Lemma \ref{L:wsh2}, 
the signed measure $[\di(S)]$ on $X$ by setting 
\begin{equation}\label{e:div1}
[\di(S)]\,\vert_{U_\alpha}=dd^c\log|s_\alpha|.
\end{equation}

Note that $[\di(\sigma^\star S)]\,\vert_{U_\alpha}=
dd^c\log|s_\alpha\circ\sigma|$. Since $\sigma^\star S$ 
is a meromorphic section of $\sigma^\star L$, 
we see that $S$ has finitely many zeros 
$z_1,\ldots,z_k\in X\setminus\Sigma$. 
By Lemma \ref{L:wsh4} we obtain  
\begin{equation}\label{e:div2}
[\di(S)]=\sigma_\star\big([\di(\sigma^\star S)]\big)=
\sum_{j=1}^k\ord(S,z_j)\delta_{z_j}+
\sum_{j=1}^m\ord(S,x_j)\delta_{x_j},
\end{equation}
where $\ord(S,x_j)$ is defined in \eqref{e:ord}.
Let $[\di(S)]^\pm$ denote the positive and negative variations 
of the measure $[\di(S)]$. Then 
\[[\di(S)]^+=\sum_{j=1}^k\ord(S,z_j)\delta_{z_j}+
\sum_{j=1}^m\ord(S,x_j)^+\delta_{x_j},\;
[\di(S)]^-=\sum_{j=1}^m\ord(S,x_j)^-\delta_{x_j},\]
where $\ord(S,x_j)^+=\max\{\ord(S,x_j),0\}$, 
$\ord(S,x_j)^-=\max\{-\ord(S,x_j),0\}$. 

Since $\di(\sigma^\star S)+R_\sigma$ is effective, we infer that 
\begin{equation}\label{e:div3}
[\di(S)]^-\leq
\sum_{j=1}^m\Big(\sum_{y\in\sigma^{-1}(x_j)}\alpha(y)\Big)
\delta_{x_j}=\sigma_\star([R_\sigma]).
\end{equation}

Note that the function $\log|S|_h\in L^1(X,\omega)$, 
since $\log|S|_h=\log|s_\alpha|-\varphi_\alpha$ on $U_\alpha$. 
Hence \eqref{e:Chern} and \eqref{e:div1} yield the following 
version of the Lelong-Poincar\'e formula in this setting:
\begin{equation}\label{e:LP}
[\di(S)]=c_1(L,h)+dd^c\log|S|_h.
\end{equation}

\medskip

The preceding discussion carries over for sections $S\in H^0_w(X,L)$, 
$S\neq0$. As $\log|s_\alpha|\in WSH(U_\alpha)$, the measure $[\di(S)]$ 
defined in \eqref{e:div1} is now positive. Moreover, we have that 
$\sigma^\star S\in H^0(\wi X,\sigma^\star L)$, so the divisor 
$\di(\sigma^\star S)$  is effective and the measure $[\di(\sigma^\star S)]$ 
is positive. Formula \eqref{e:div2} and the Lelong-Poincar\'e formula \eqref{e:LP} 
hold for $[\di(S)]$.

\medskip

We give now the definitions of the Bergman kernel functions 
$P_{w,p}, P_{c,p}, P_p$ and Fubini-Study measures 
$\gamma_{w,p}, \gamma_{c,p}, \gamma_p$ of the Bergman spaces 
considered in this paper. We start with the spaces $H^0_{w,(2)}(X,L^p)$ 
and $H^0_{c,(2)}(X,L^p)$ defined in \eqref{e:Bsw}, respectively \eqref{e:Bsc}. 

Let $d_{w,p}=\dim H^0_{w,(2)}(X,L^p)$ and let $S^p_j$, 
$1\leq j\leq d_{w,p}$, be an orthonormal basis of $H^0_{w,(2)}(X,L^p)$. 
We write $S^p_j=s^p_{j,\alpha}e_\alpha^{\otimes p}$, 
where $s^p_{j,\alpha}\in\cO_{X,w}(U_\alpha)$. Then 
\begin{equation}\label{e:BFS1}
P_{w,p}(x)=\sum_{j=1}^{d_{w,p}}|S^p_j(x)|_{h_p}^2\;,\;\;
\gamma_{w,p}\,\vert_{U_\alpha}=dd^cu_{p,\alpha}, 
\text{ where } u_{p,\alpha}=\frac{1}{2}\,
\log\left(\sum_{j=1}^{d_{w,p}}|s^p_{j,\alpha}|^2\right).
\end{equation}
Note that $P_{w,p}(x)$ is defined for all $x\in X\setminus\Sigma$ 
such that if $x\in U_\alpha$ then $\varphi_\alpha(x)>-\infty$. 
Moreover, $u_{p,\alpha}\in WSH(U_\alpha)$ and $u_{p,\alpha}=
u_{p,\beta}+\log|g_{\alpha\beta}|$ on 
$(U_\alpha\cap U_\beta)\setminus\Sigma$, 
so by Lemma \ref{L:wsh2} $\gamma_{w,p}$ 
is a well defined positive measure on $X$. 
We have that $P_{w,p}$, $\gamma_{w,p}$ 
are independent of the choice of basis and 
\begin{equation}\label{e:BFS2}
\log P_{w,p}\,\vert_{U_\alpha}=2u_{p,\alpha}-2p\varphi_\alpha.
\end{equation}
By Lemma \ref{L:wsh1} we infer that $\log P_{w,p}\in L^1(X,\omega)$,
and 
\begin{equation}\label{e:BFS3}
\frac{1}{p}\,\gamma_{w,p}-c_1(L,h)=\frac{1}{2p}\,dd^c\log P_{w,p}\,.
\end{equation}
Moreover, as in \cite{CM15,CM13,CMM17}, 
one has the following variational formula,  
\begin{equation}\label{e:Bkvar}
P_{w,p}(x)=\max\big\{|S(x)|^2_{h_p}:\,S\in H^0_{w,(2)}(X,L^p),\;\|S\|_p=1\big\},
\end{equation}
for all $x\in X\setminus\Sigma$ where $P_{w,p}(x)$ is defined. 

Let $d_{c,p}=\dim H^0_{c,(2)}(X,L^p)$. 
One defines the Bergman kernel function $P_{c,p}$ 
and Fubini-Study measure $\gamma_{c,p}$ of the space 
$H^0_{c,(2)}(X,L^p)$ in the same way as in \eqref{e:BFS1}. 
The formulas \eqref{e:BFS2}, \eqref{e:BFS3}, \eqref{e:Bkvar} 
hold in this case as well.

\medskip

We next turn our attention to the Bergman kernel function $P_p$ 
and Fubini-Study measure $\gamma_p$ of the space 
$H^0_{(2)}(X\setminus\Sigma,L^p)$ defined in \eqref{e:Bsreg}. 
Proceeding as above, let $d_p=\dim H^0_{(2)}(X\setminus\Sigma,L^p)$ 
and $S^p_j$, $1\leq j\leq d_p$, be an orthonormal basis of 
$H^0_{(2)}(X\setminus\Sigma,L^p)$. 
We set  $S^p_j=s^p_{j,\alpha}e_\alpha^{\otimes p}$, 
where $s^p_{j,\alpha}\in\cO_X(U_\alpha\setminus\Sigma)$, 
and define for $x\in X\setminus\Sigma$,
\begin{equation}\label{e:BFS4}
P_p(x)=\sum_{j=1}^{d_p}|S^p_j(x)|_{h_p}^2\;,\;
\text{ and $\gamma_p\,\vert_{U_\alpha}=
dd^cu_{p,\alpha}$, where } u_{p,\alpha}=
\frac{1}{2}\,\log\left(\sum_{j=1}^{d_p}|s^p_{j,\alpha}|^2\right).
\end{equation}
Proposition \ref{P:Bsreg} shows that $u_{p,\alpha}\circ\sigma$ 
is locally the difference of two subharmonic functions on 
$\sigma^{-1}(U_\alpha)$. Hence $u_{p,\alpha}$ is 
locally the difference of two weakly subharmonic functions on $U_\alpha$. 
By Lemma \ref{L:wsh1} and  Lemma \ref{L:wsh2} we infer that 
$u_{p,\alpha}\in L^1_{loc}(U_\alpha,\omega)$ and $\gamma_p$ 
is a well defined signed measure on $X$. The analogue of \eqref{e:BFS2} 
in this setting shows that 
\begin{equation}\label{e:BFS5}
\log P_p\in L^1(X,\omega)\,,\,\;\frac{1}{p}\,\gamma_p-c_1(L,h)=
\frac{1}{2p}\,dd^c\log P_p\,.
\end{equation}
Moreover, the variational formula \eqref{e:Bkvar} also holds for $P_p$.

\begin{Lemma}\label{L:FSneg}
Let $\gamma_p^-$ be the negative variation of the measure $\gamma_p$. 
Then $\gamma_p^-\leq\sigma_\star([R_\sigma])$. In particular, 
$\gamma_p^-$ is supported in $\Sigma$.
\end{Lemma}

\begin{proof}
Since $u_{p,\alpha}\circ\sigma$ is locally the difference of two 
subharmonic functions on $\sigma^{-1}(U_\alpha)$, 
$\wi\gamma_p\,\vert_{\sigma^{-1}(U_\alpha)}:=
dd^c(u_{p,\alpha}\circ\sigma)$ defines a signed measure 
$\wi\gamma_p$ on $\wi X$. By Lemma \ref{L:wsh4} we infer that 
$\gamma_p=\sigma_\star(\wi\gamma_p)$. 
Note that $\wi X\setminus\sigma^{-1}(\Sigma)$ 
is a positive set for $\wi\gamma_p$. 
Working locally near a point 
$y\equiv0\in\sigma^{-1}(\Sigma\cap U_\alpha)$ 
we have by Proposition \ref{P:Bsreg} that 
$u_{p,\alpha}(\sigma(\zeta))=v(\zeta)+n\log|\zeta|$, 
where $v$ is a smooth subharmonic function and 
$n\geq-\alpha(y)$. Hence $\wi\gamma_p(\{y\})=n\geq-\alpha(y)$. 
Therefore the measure $\wi\gamma_p+[R_\sigma]$ is positive, 
hence so is $\gamma_p+\sigma_\star([R_\sigma])$. 
The conclusion now follows.
\end{proof}

We conclude this section by noting that the spaces 
\begin{equation}\label{e:Bscomp}
H^0_{c,(2)}(X,L^p)\subset H^0_{w,(2)}(X,L^p)\subset 
H^0_{(2)}(X\setminus\Sigma,L^p),
\end{equation} 
are endowed with the same inner product, hence 
\begin{equation}\label{e:Bkcomp}
P_{c,p}(x)\leq P_{w,p}(x)\leq P_p(x),
\end{equation}
for every $x\in X\setminus\Sigma$  such that if 
$x\in U_\alpha$ then $\varphi_\alpha(x)>-\infty$.

\section{Proof of Theorems \ref{T:mt} and \ref{T:zero}}\label{S:proofs}

We start with some auxiliary results that are needed for the proof of Theorem \ref{T:mt}.

\begin{Lemma}\label{L:sh}
If $v$ is subharmonic on $\D$, $A=\{\zeta\in\C:\,\frac{1}{8}\leq|\zeta|\leq\frac{5}{8}\}$ and $|x|\leq\frac{1}{8}$, then 
\[v(x)\leq2\int_A|v|\,d\lambda.\]
\end{Lemma}

\begin{proof}
Let $A_x=\{\zeta\in\C:\,\frac{1}{4}\leq|\zeta-x|\leq\frac{1}{2}\}$. Since $A_x\subset A$ we obtain using the subaverage inequality that
\[\frac{3}{32}\,v(x)\leq\frac{1}{2\pi}\,\int_\frac{1}{4}^\frac{1}{2}\int_0^{2\pi}v(x+re^{it})\,rdtdr\leq\frac{1}{2\pi}\,\int_{A_x}|v|\,d\lambda\leq\frac{1}{2\pi}\,\int_A|v|\,d\lambda.\]
\end{proof}

We state next a general result about the asymptotics of the logarithms of Bergman kernels. Let $Y$ be a complex manifold of dimension $n$, $\omega$ be a Hermitian form on $Y$ and $(L,h)$ be a singular Hermitian holomorphic line bundle on $Y$. One defines the Bergman spaces $H^0_{(2)}(Y,L^p)$ in analogy to \eqref{e:Bsreg} using the metric $h_p$ induced by $h$ on $L^p$ and the volume form $\frac{\omega^n}{n!}$ on $Y$. Let $V_p\leq H^0_{(2)}(Y,L^p)$ be a subspace of dimension $m_p$ and define the Bergman kernel function $Q_p$ of $V_p$ as in \eqref{e:BFS4}. Note that $\log Q_p$ is locally the difference of integrable functions, so $\log Q_p\in L^1_{loc}(Y,\omega^n)$.

\begin{Lemma}\label{L:Bkf}
In the above setting, assume that:

\noindent
(i) $m_p\leq cp^n$ holds for all $p\geq1$, with some constant $c>0$; 

\noindent
(ii) every $x\in Y$ has a neighborhood $U_x$ such that $Q_p\geq\varepsilon_x$ holds $\omega^n$-a.e.\ on $U_x$ for all $p$ sufficiently large, with some constant $\varepsilon_x>0$. 

Then $\frac{1}{p}\log Q_p\to0$ as $p\to+\infty$, in $L^1_{loc}(Y,\omega^n)$. 
\end{Lemma}

\begin{proof}
We may assume that $C_x:=\int_{U_x}\omega^n<+\infty$. It suffices to show that $\frac{1}{p}\log Q_p\to0$ in $L^1(U_x,\omega^n)$ for each $x\in Y$. Note that $\int_YQ_p\,\frac{\omega^n}{n!}=m_p$. Using Jensen's inequality and hypothesis $(i)$ we get 
\begin{equation}\label{e:logQ}
\int_{U_x}(\log Q_p)\,\omega^n\leq C_x\log\Big(\frac{1}{C_x}\int_{U_x}Q_p\,\omega^n\Big)\leq nC_x\log p+C'_x
\end{equation}
for all $p\geq1$, with some constant $C'_x>0$. If $f_p:=\frac{1}{p}\,(\log Q_p-\log\varepsilon_x)$ then by $(ii)$, $f_p\geq0$ $\omega^n$-a.e.\ on $U_x$ for all $p$ sufficiently large. Hence $\int_{U_x}f_p\,\omega^n\to0$ by \eqref{e:logQ}. This implies the conclusion.
\end{proof}

\smallskip

We now introduce the geometric setting and notation needed to describe the local structure of $X$ near the singular points. Let $\sigma:\wi X\to{\mathbb P}^N$ be the normalization of $X$ from (C). For each $x_j\in\Sigma$ we can find a neighborhood $V_j\subset X$ of $x_j$ with the following properties: 

\smallskip

(P1) $\ov V_j\cap\ov V_k=\emptyset$ for $1\leq j<k\leq m$.

\smallskip

(P2) there exist coordinates $(z_1,\ldots,z_N)$ on $\C^N\hookrightarrow{\mathbb P}^N$ and a polydisc $D_j\subset\C^N$ centered at $0$ such that $x_j\equiv 0$, $V_j=X\cap D_j$, and $V_j\cap\{z_1=0\}=\{0\}$.

\smallskip

(P3) $Y_1,\ldots,Y_{k_j}$ denote the irreducible components of $V_j$. 

\smallskip

(P4) for $1\leq\ell\leq k_j$, there exists $y_\ell\in\sigma^{-1}(x_j)$ and a coordinate neighborhood of $y_\ell\equiv 0$ containing $\D$, such that $\sigma_\ell:=\sigma\,\vert_\D:\D\to D_j$ is a local normalization of $Y_\ell$ as in \eqref{e:ln}. 

\smallskip

Note that by (P1), any two irreducible components of $V_j$ intersect only at $x_j$. In the proof of Theorem \ref{T:mt} we will have to work with a different Hermitian form on $X\setminus\Sigma$, which is provided by the following lemma.

\begin{Lemma}\label{L:Om}
Let $\Omega:=\sigma_\star\Big(\wi\omega\,\vert_{\wi X\setminus\sigma^{-1}(\Sigma)}\Big)$. Then $\Omega$ is a Hermitian form which verifies $\ric_\Omega\geq-2\pi B\Omega$ and $\Omega\geq a\omega$ on $X\setminus\Sigma$, with some constants $a,B>0$.
\end{Lemma}

\begin{proof}
As $\wi X$ is compact, there exists a constant $B>0$ such that $\ric_{\wi\omega}\geq-2\pi B\wi\omega$ on $\wi X$. Since $\sigma:\wi X\setminus\sigma^{-1}(\Sigma)\to X\setminus\Sigma$ is biholomorphic, $\Omega$ is a Hermitian form with $\ric_\Omega\geq-2\pi B\Omega$ on $X\setminus\Sigma$. For $x_j\in\Sigma$, let $Y_\ell$ be a component of $X$ at $x_j$ as in (P3), and $\sigma_\ell$ be as in (P4). We infer from \eqref{e:ram2} that $\sigma_\ell^\star\omega\leq C\wi\omega$ holds on $\D$, with some constant $C>0$. Hence $\omega\leq C\Omega$ on $Y_\ell\setminus\{x_j\}$. This yields the conclusion.
\end{proof}

The function $\phi$ constructed in the next lemma will be used to obtain a suitable modified metric on $L$. 

\begin{Lemma}\label{L:phi}
Let $X,\Sigma,\omega,\sigma$ verify (A), (C) and $\Omega$ be as in Lemma \ref{L:Om}. If $U\supset\Sigma$ is an open set, then there exists a continuous function $\phi:X\to\R$ supported in $U$ such that $\phi\leq0$ on $X$, $\phi$ is smooth on $X\setminus\Sigma$, and $A\omega+dd^c\phi\geq b\Omega$ on $X\setminus\Sigma$ for some constants $A,b>0$.
\end{Lemma}

\begin{proof}
For $1\leq j\leq m$, we choose a neighborhood $V_j\subset U$ of $x_j$ such that properties (P1)-(P4) are satisfied. Fix now $x_j\in\Sigma$. Since $V_j\cap\{z_1=0\}=\{0\}$, by shrinking $D_j$ we may assume that for each $Y_\ell$ we have 
\[\sigma_\ell(\zeta)=(\zeta^{s_\ell},f_{\ell,2}(\zeta),\ldots,f_{\ell,N}(\zeta)),\;\zeta\in\D,\] 
where $s_\ell-1\geq\alpha(y_\ell)$ (see \eqref{e:ram1}, \eqref{e:ramdiv}) and $f_{\ell,l}$ are holomorphic in $\D$. We consider the plurisubharmonic function $w_\ell(z_1,\ldots,z_N)=\log\big(1+|z_1|^{2/s_\ell}\big)$ on $\C^N$, and define 
\[\rho_j=\sum_{\ell=1}^{k_j}w_\ell\,\vert_{V_j}-M_j,\] 
where $M_j$ is chosen so that $\rho_j<0$ on $V_j$. Then $\rho_j$ is continuous and subharmonic on $V_j$, and it is smooth on $V_j\setminus\{0\}$ since $z_1\neq0$ if $(z_1,\ldots,z_N)\in V_j\setminus\{0\}$. Moreover, 
\begin{equation}\label{e:wj}
dd^c(\rho_j\circ\sigma_\ell)\geq dd^c(w_\ell\circ\sigma_\ell)=dd^c\log(1+|\zeta|^2)\geq\frac{i}{4\pi}\,d\zeta\wedge d\ov\zeta
\end{equation}
holds on $\D$, for each $\ell=1,\ldots,k_j$. 

Let $\chi_j\in\cC_{X,0}^\infty(V_j)$ be such that $\chi_j\geq0$ on $X$ and $\chi_j=1$ on an open set $W_j$ containing $x_j$. We define 
\[\phi=\sum_{j=1}^m\chi_j\rho_j.\] 
Then $\phi\leq0$ is continuous on $X$, $\supp\phi\subset U$, and $\phi$ is smooth on $X\setminus\Sigma$. Since $\phi=\rho_j$ on $W_j$, $\phi$ is subharmonic in the neighborhood $W:=\bigcup_{j=1}^m W_j$ of $\Sigma$. It follows that there exists a constant $M>0$ such that $dd^c\phi\geq-M\Omega$ on $X\setminus\Sigma$. Since $\phi=\rho_j$ on $W_j$ we infer by \eqref{e:wj} that $dd^c(\phi\circ\sigma)\geq b\wi\omega$ holds on $\sigma^{-1}(W)$ for some constant $b>0$, hence $dd^c\phi\geq b\Omega$ on $W\setminus\Sigma$. Let $c>0$ be a constant such that $\omega\geq c\Omega$ on $X\setminus W$. For $A>0$ we obtain 
\begin{align*}
A\omega+dd^c\phi & \geq(Ac-M)\Omega\, \text{ on } X\setminus W, \\
A\omega+dd^c\phi & \geq b\Omega\, \text{ on } W\setminus\Sigma.
\end{align*}
The conclusion follows by choosing $A=(M+b)/c$.
\end{proof}

\smallskip

\begin{proof}[Proof of Theorem \ref{T:mt}] $(i)$ We show first that 
\begin{equation}\label{e:creg}
\frac{1}{p}\,\log P_{c,p}\to0,\;\frac{1}{p}\,
\log P_{w,p}\to0,\;\frac{1}{p}\,\log P_p\to0,\, 
\text{ in }L^1_{loc}(X\setminus\Sigma,\omega).
\end{equation}

By Proposition \ref{P:Bsreg}, \eqref{e:Bscomp} and \eqref{e:Bkcomp}, 
this will follow from Lemma \ref{L:Bkf} once we show that hypothesis 
$(ii)$ of Lemma \ref{L:Bkf} holds for $Q_p=P_{c,p}$. 
Let $x\in X\setminus\Sigma$ and $U_\alpha$ be a coordinate neighborhood 
of $x\equiv0$ such that $\ov U_\alpha\subset X\setminus\Sigma$ and $L$ 
has a holomorphic frame $e_\alpha$ on $U_\alpha$. 
Set $|e_\alpha|_h=e^{-\varphi_\alpha}$, 
so $\varphi_\alpha\in SH(U_\alpha)$ by (B). 

For each $x_j\in\Sigma$ we choose a neighborhood 
$V_j\subset X\setminus\ov U_\alpha$ of $x_j$ 
such that properties (P1)-(P4) are satisfied. 
We define the function $\rho_j$ on $V_j$ by setting 
\begin{equation}\label{e:rho} 
\rho_j=\big(\alpha(y_\ell)+1\big)\log|\sigma_\ell^{-1}| 
\text{ on } Y_\ell\setminus\{x_j\},\;\ell=1,\ldots,k, \text{ and } 
\rho_j(x_j)=-\infty,
\end{equation}
where $\alpha(y_\ell)$ is the ramification index of $\sigma$ at $y_\ell$. 
By Lemma \ref{L:wsh3} it follows that $\rho_j\in SH(V_j)$. 
Moreover $\rho_j<0$ and $\rho_j$ is smooth on $V_j\setminus\{x_j\}$. 
Let $\chi_j\in\cC_{X,0}^\infty(V_j)$ be such that $\chi_j\geq0$ 
on $X$ and $\chi_j=1$ on an open set $W_j$ containing $x_j$. 
We define $\eta=\sum_{j=1}^m\chi_j\rho_j$. 
Then $\eta$ is smooth on $X\setminus\Sigma$, $\eta\leq0$ on $X$, 
and $\eta=0$ in a neighborhood of $\ov U_\alpha$. 
Since $\eta=\rho_j$ on $W_j$, $\eta$ is subharmonic 
in a neighborhood of $\Sigma$. We infer that 
\begin{equation}\label{e:eta}
dd^c\eta\geq-M\Omega \text{ on } X\setminus\Sigma,
\end{equation}
for some constant $M>0$, where $\Omega$ is as in Lemma \ref{L:Om}.  

Fix $r_0>0$ such that the disc $V:=\D_{2r_0}\Subset U_\alpha$ 
and we set $U:=\D_{r_0}$. We will show that there exist a constant 
$C>0$ and $p_0\in\mathbb{N}$ with the following property: 
if $p>p_0$, $z\in U$ and $\varphi_\alpha(z)>-\infty$, 
then there exists a section $S_{z,p}\in H^0_{c,(2)}(X,L^p)$ 
such that $S_{z,p}(z)\neq0$ and
\begin{equation}\label{e:lest}
\|S_{z,p}\|^2_p\leq C|S_{z,p}(z)|^2_{h_p}\,.
\end{equation}
In view of \eqref{e:Bkvar} this implies that 
\[P_{c,p}(z)\geq\frac{|S_{z,p}(z)|^2_{h_p}}{\|S_{z,p}\|^2_p}\geq C^{-1},\]
which shows that hypothesis $(ii)$ of Lemma \ref{L:Bkf} folds for $P_{c,p}$.

For the proof of \eqref{e:lest} we use techniques of Demailly \cite[Section 9]{D93b} (see also \cite[Section 5]{CM15}). By the Ohsawa-Takegoshi extension theorem \cite{OT87} there exists a constant $C'>0$ (depending only on $x$) such that for any $z\in U$ with $\varphi_\alpha(z)>-\infty$ and any $p$ there exists a function $v_{z,p}\in\cO_X(V)$ with $v_{z,p}(z)\neq0$ and
\[\int_V|v_{z,p}|^2e^{-2p\varphi_\alpha}\Omega\leq C'|v_{z,p}(z)|^2e^{-2p\varphi_\alpha(z)}.\]
We will extend  $v_{z,p}$ to a section $S_{z,p}\in H^0_{c,(2)}(X,L^p)$ by solving a $\db$-equation with $L^2$-estimates. 
Let $\chi:[0,+\infty)\to[0,1]$ be a smooth function such that $\chi=1$ on $[0,\tfrac{1}{2}]$ and $\chi=0$ on $[\tfrac{3}{4},+\infty)$, and set 
\[\theta_z(t)=\begin{cases}
\chi\big(\frac{|t-z|}{r_0}\big)\log\frac{|t-z|}{r_0}, & \text{for $t\in U_\alpha$},\\
0, & \text{for $t\in X\setminus U_\alpha$}.
\end{cases}\]

Let $\phi$ be the function constructed in Lemma \ref{L:phi} 
for the open set $X\setminus\ov U_\alpha\supset\Sigma$, 
so $\phi\leq0$ on $X$, $\phi=0$ in a neighborhood of 
$\ov U_\alpha$, and $dd^c\phi\geq b\Omega-A\omega$ on 
$X\setminus\Sigma$ for some constants $A,b>0$. 
We consider the metric 
\[\wi h_p=h_pe^{-2\delta p\phi-2\eta-2\theta_z}\,
\text{ on } L^p\,\vert_{X\setminus\Sigma}\,,\,
\text{ where } \delta=\frac{\varepsilon}{A}\,.\] 
Since $\theta_z$ is subharmonic in a neighborhood of $z$, 
it follows that there exists a constant $M'>0$ such that 
$dd^c\theta_z\geq-M'\Omega$ on $X\setminus\Sigma$, 
for all $z\in U$. Using Lemma \ref{L:phi}, \eqref{e:eta} 
and hypothesis (B), we obtain for all $p$ sufficiently large that 
\begin{align*}c_1(L^p,\wi h_p) &=pc_1(L,h)+\delta p\,
dd^c\phi+dd^c\eta+dd^c\theta_z \\
&\geq p\Big(c_1(L,h)+\frac{\varepsilon b}{A}\,
\Omega-\varepsilon\omega\Big)-(M+M')\Omega \\
&\geq\Big(\frac{p\varepsilon b}{A}-M-M'\Big)\Omega\geq 
2p\varepsilon'\,\Omega
\end{align*}
on $X\setminus\Sigma$, where $\varepsilon'=
\frac{\varepsilon b}{4A}$. Note that the Riemann surface 
$X\setminus\Sigma$ is Stein (see e.g.\ \cite[Corollary 26.8]{For81}), 
hence it carries a complete K\"ahler metric. By Lemma \ref{L:Om} we have 
$\ric_\Omega\geq-2\pi B\Omega$ on $X\setminus\Sigma$. Let
\[g\in L^2_{0,1}(X\setminus\Sigma,L^p,{\rm loc}),\;g=
\overline\partial\Big(v_{z,p}\chi\big(\tfrac{|t-z|}{r_0}\big)
e_\alpha^{\otimes p}\Big).\]
By \cite[Theorem 5.1]{D82} (see also \cite[Theorem 2.5]{CMM17}) 
it follows that there exists $p_0\in\N$ such that if $p>p_0$ 
there exists $u\in L^2_{0,0}(X\setminus\Sigma,L^p,{\rm loc})$ 
verifying $\db u =g$ and
\[\int_{X\setminus\Sigma}|u|^2_{h_p}e^{-2\delta p\phi-2\eta-2\theta_z}
\Omega\leq\frac{1}{p\varepsilon'}
\int_{X\setminus\Sigma}|g|^2_{h_p}e^{-2\delta p\phi-
2\eta-2\theta_z}\Omega.\]
Since $\phi=\eta=0$ on $U_\alpha$ we have that 
\begin{align*}
\int_{X\setminus\Sigma}|g|^2_{h_p}e^{-2\delta p\phi-
2\eta-2\theta_z}\Omega & =\int\limits_{\{\frac{r_0}{2}<
|t-z|<r_0\}}|v_{z,p}|^2|
\db\chi(\tfrac{|t-z|}{r_0})|^2
e^{-2p\varphi_\alpha-2\theta_z}\Omega \\ 
& \leq C''\int_V|v_{z,p}|^2e^{-2p\varphi_\alpha}
\Omega\leq C'C''|v_{z,p}(z)|^2e^{-2p\varphi_\alpha(z)}<+\infty,
\end{align*}
where $C''>0$ is a constant depending only on $x$. 
Using $\phi\leq0$, $\eta\leq0$, $\theta_z\leq0$, we get  
\[\int_{X\setminus\Sigma}|u|^2_{h_p}\Omega\leq
\int_{X\setminus\Sigma}|u|^2_{h_p}e^{-2\eta-2\theta_z}
\Omega\leq\frac{C'C''}{p\varepsilon'}\,|v_{z,p}(z)|^2
e^{-2p\varphi_\alpha(z)}.\]

Note that $u(z)=0$, as $e^{-2\theta_z(t)}=r_0^2|t-z|^{-2}$ 
is not integrable near $z$. Fix now $x_j\in\Sigma$. 
We have $\theta_z=0$ on $V_j$ and $\eta=\rho_j$ on $W_j$. 
Let $Y_\ell$ be a component of $V_j$ as in (P3) and 
$\sigma_\ell:\D\to Y_\ell$ be the normalization of $Y_\ell$ as in (P4). 
We may assume that $L$ has a holomorphic frame 
$e_L$ on $V_j$ and that the corresponding local weight of $h$ 
is upper bounded on $V_j$. Set $u=ve_L^{\otimes p}$. 
We infer that there exists a constant $c>0$ such that 
\[\int_{Y_\ell\setminus\{x_j\}}|u|^2_{h_p}e^{-2\eta-2\theta_z}\Omega 
\geq c^p\int_{(Y_\ell\cap W_j)\setminus\{x_j\}}|v|^2e^{-2\rho_j}\Omega
\geq ac^p\int_{(Y_\ell\cap W_j)\setminus\{x_j\}}|v|^2e^{-2\rho_j}\omega,\]
where the second inequality follows from Lemma \ref{L:Om}. 
Using \eqref{e:rho} and \eqref{e:ram2} we get  
\[\int_{(Y_\ell\cap W_j)\setminus\{x_j\}}|v|^2e^{-2\rho_j}
\omega\geq C_2^{-1}\int_{\D_r\setminus\{0\}}|v(\sigma_\ell(\zeta))|^2
|\zeta|^{-2}\,d\lambda,\]
for some $r\in(0,1)$. Therefore 
$\int_{\D_r\setminus\{0\}}|v(\sigma_\ell(\zeta))|^2|\zeta|^{-2}\,
d\lambda<+\infty$. Note that $\db u=g=0$ on $V_j\setminus\{x_j\}$, 
hence $v\circ\sigma_\ell$ is holomorphic on $\D\setminus\{0\}$. 
We infer from Lemma \ref{L:Ls} that $v\circ\sigma_\ell$ extends 
holomorphically at $0$ and $v\circ\sigma_\ell(0)=0$. 
This shows that $u$ is a continuous weakly holomorphic section of 
$L^p$ on $V_j$, for $j=1,\ldots,m$.

Set $S_{z,p}:=v_{z,p}\chi\big(\tfrac{|t-z|}{r_0}\big)e_\alpha^{\otimes p}-u$. 
Then $\db S_{z,p}=0$, $S_{z,p}=-u$ on each $V_j$, so 
$S_{z,p}\in H^0_c(X,L^p)$. Moreover, $S_{z,p}(z)=
v_{z,p}(z)e_\alpha^{\otimes p}(z)\neq0$, as $u(z)=0$. 
Using Lemma \ref{L:Om} we obtain
\begin{align*}
\|S_{z,p}\|^2_p  & \leq\frac{1}{a}\,
\int_{X\setminus\Sigma}|S_{z,p}|^2_{h_p}\Omega\leq
\frac{2}{a}\,\left(\int_V|v_{z,p}|^2e^{-2p\varphi_\alpha}\Omega+
\int_{X\setminus\Sigma}|u|^2_{h_p}\Omega\right) \\
& \leq\frac{2C'}{a}\left(1+\frac{C''}{p\varepsilon'}\right)
|v_{z,p}(z)|^2e^{-2p\varphi_\alpha(z)}\leq C|S_{z,p}(z)|^2_{h_p},
\end{align*}
with a constant $C>0$ depending only on $x$. 
This concludes the proof of \eqref{e:lest}, and hence of \eqref{e:creg}. 

\smallskip

We consider next the $L^1$-convergence near the singular points. 
Let $V_j$, $1\leq j\leq m$, verify properties (P1)-(P4), and fix 
$x_j\in\Sigma$. Let $Y_\ell$ be a component of $V_j$ as in (P3) 
and $\sigma_\ell$ be as in (P4). We may assume that $L$ has a 
holomorphic frame $e_L$ on $V_j$ and that the corresponding 
local weight $\varphi$ of $h$ is upper bounded on $V_j$. 
The proof of assertion $(i)$ is complete if we show that  
\begin{equation}\label{e:csing}
\frac{1}{p}\,\log P_{c,p}\to0,\;\frac{1}{p}\,
\log P_{w,p}\to0,\;\frac{1}{p}\,\log P_p\to0,\, \text{ in }L^1(Y_\ell,\omega).
\end{equation}

Let $Q_p$ denote either one of the Bergman kernels 
$P_{c,p},P_{w,p},P_p$, and let $S^p_1,\ldots,S^p_{n_p}$ 
be an orthonormal basis of the corresponding Bergman space. 
We write 
\[S^p_j=s^p_je_L^{\otimes p}, \text{ where } s^p_j\in
\cO_X(V_j\setminus\{x_j\}),\,\;u_p:=\frac{1}{2}\,
\log\left(\sum_{j=1}^{n_p}|s^p_j|^2\right).\]
Formula \eqref{e:BFS2} implies that 
\[\frac{1}{p}\,u_p\circ\sigma_\ell-\varphi\circ\sigma_\ell=
\frac{1}{2p}\,\log(Q_p\circ\sigma_\ell)\,
\text{ on } \D\setminus\{0\}.\]
By Proposition \ref{P:Bsreg} and \eqref{e:Bscomp}, 
we have that the functions $\zeta^{\alpha(y_\ell)}s^p_j(\zeta)$ 
extend holomorphically at $0\in\D$. Hence the function 
\[v_p(\zeta):=u_p\circ\sigma_\ell(\zeta)+\alpha(y_\ell)\log|\zeta|\]
extends to a subharmonic on $\D$. Moreover, $\varphi\circ\sigma_\ell$ 
also extends to a subharmonic on $\D$.

We infer from \eqref{e:creg} that 
$\frac{1}{p}\,\log(Q_p\circ\sigma_\ell)\to0$, hence 
$\frac{1}{p}\,v_p\to\varphi\circ\sigma_\ell$, in 
$L^1_{loc}(\D\setminus\{0\},\lambda)$. 
Combined with Lemma \ref{L:sh} this implies that the 
sequence of subharmonic functions $\{\frac{1}{p}\,v_p\}$ 
is locally uniformly upper bounded in $\D$. 
Therefore \cite[Theorem 3.2.12]{H94} yields that 
$\frac{1}{p}\,v_p\to\varphi\circ\sigma_\ell$, 
hence $\frac{1}{p}\,\log(Q_p\circ\sigma_\ell)\to0$, 
in $L^1_{loc}(\D,\lambda)$. Using \eqref{e:ram2} we get

\[\int_{Y_\ell\setminus{x_j}}|\log Q_p|\,\omega\leq 
C_2\int_\D|\log(Q_p\circ\sigma_\ell)|\,d\lambda.\]
Hence \eqref{e:csing} follows, and the proof of assertion $(i)$ is finished.

\medskip

$(ii)$ In the case of $\gamma_{c,p}$ and $\gamma_{w,p}$, $(ii)$ follows immediately from $(i)$ by using \eqref{e:BFS3}. From \eqref{e:BFS5} we see that $\frac{1}{p}\,\int_X\chi\,d\gamma_p\to\int_X\chi\,dc_1(L,h)$ for every smooth function $\chi$ on $X$. By Lemma \ref{L:FSneg} the measure $\gamma_p+\sigma_\star([R_\sigma])$ is positive, so we infer that  $\frac{1}{p}\,\gamma_p\to c_1(L,h)$ in the weak sense of measures. Therefore $\frac{1}{p}\,\gamma_p^+\to c_1(L,h)$, as $\frac{1}{p}\,\gamma_p^-\to0$ in the weak sense of measures, by Lemma \ref{L:FSneg}.
\end{proof}

\smallskip

\begin{proof}[Proof of Theorem \ref{T:zero}]
Let $Q_p,\eta_p$ be the Bergman kernel function and 
Fubini-Study measure of the space $H^p$. 
By Theorem 1.1 we have $\frac{1}{p}\,\log Q_p\to0$ in $L^1(X,\omega)$ 
and $\frac{1}{p}\,\eta_p\to c_1(L,h)$ in the weak sense of measures on $X$. 
Note that $c_1(L^p,h_p)=pc_1(L,h)$ and formulas 
\eqref{e:BFS3}, \eqref{e:BFS5}, which relate $\eta_p,Q_p,c_1(L,h)$, 
are valid. Moreover, we have the Lelong-Poincar\'e formula \eqref{e:LP} 
relating $[\di(s_p)],|s_p|_{h_p},c_1(L,h)$. 
Hence the proof of \cite[Theorem 1.1]{BCM20} 
goes through with no change in our setting, and we can take $A_p=p$ 
(see also \cite[Theorem 4.1]{BCM20} and its proof).

In our present situation $\mu_p$ is the normalized 
area measure on the unit sphere of $H^p$. 
By Proposition \ref{P:Bsreg} we have $n_p=\dim H^p\leq Cp$ 
for some constant $C>0$. Therefore Theorem \ref{T:zero} 
follows from \cite[Theorem 4.12]{BCM20}. 

We have to consider further the case 
$H^p=H^0_{(2)}(X\setminus\Sigma,L^p)$, 
when the measures $\eta_p=\gamma_p$ and $[\di(s_p)]$ 
are not necessarily positive. 
Arguing as in the proof of \cite[Theorem 1.1]{BCM20}, 
we have that $\displaystyle\frac{1}{p}\,\log|s_p|_{h_p}\to0$ 
in $L^1(X,\omega)$ for $\mu$-a.\,e.\ sequence $\{s_p\}\in\mathcal{H}$. 
By \eqref{e:LP}, this implies that  $\frac{1}{p}\,[\di(s_p)]\to c_1(L,h)$ 
in the sense of distributions. Since by \eqref{e:div3}, 
$[\di(s_p)]^-\leq\sigma_\star([R_\sigma])$, 
it follows that $\frac{1}{p}\,[\di(s_p)]\to c_1(L,h)$, 
$\frac{1}{p}\,[\di(s_p)]^+\to c_1(L,h)$, 
in the weak sense of measures on $X$. 
This completes the proof of assertion $(ii)$.

For $(i)$, we have as in the proof of \cite[Theorem 1.1]{BCM20} 
that $\E[\di(s_p)]$ is a well defined distribution on $X$ and 
\[\Big\langle\frac{1}{p}\,\E[\di(s_p)],
\chi\Big\rangle\to\int_X\chi\,dc_1(L,h)\]
for every smooth function $\chi$. By \eqref{e:div3} 
the measure $\nu:=[\di(s_p)]+\sigma_\star([R_\sigma])$ 
is positive, hence the total variation $|[\di(s_p)]|\leq\nu+
\sigma_\star([R_\sigma])$. Using \eqref{e:LP} we get 
\[\Big|\int_X\chi\,d[\di(s_p)]\Big|\leq\int_X|\chi|\,
d(\nu+\sigma_\star([R_\sigma])\leq
\|\chi\|_\infty\big(pc_1(L,h)+2\sigma_\star([R_\sigma])\big)(X).\]
We infer that $\E[\di(s_p)]$ is a well defined measure on $X$ 
and its total variation verifies 
\[\big|\E[\di(s_p)]\big|(X)\leq\big(pc_1(L,h)+
2\sigma_\star([R_\sigma])\big)(X).\]
This yields assertion $(i)$.
\end{proof}

\section{Examples}\label{S:ex}

In this section we exemplify our results
in the case of certain plane algebraic curves.
We also give a precise lower estimate
of the Bergman kernel $P_{w,p}$
in the case of a smooth Hermitian metric on $L$.

\begin{Example}
In this section we consider a class of algebraic curves $X$ 
in ${\mathbb P}^2$ which have one singular point 
and are normalized by ${\mathbb P}^1$. 
They are defined by graphs of polynomials and 
we will describe explicitly the spaces of holomorphic sections 
considered in the paper. We denote by $\C_n[\zeta]$ 
the space of polynomials of degree at most $n$ in $\C$.

Let $[z_0:z_1:z_2]$ denote the homogeneous coordinates on 
${\mathbb P}^2$, and consider the standard embedding 
$(z_1,z_2)\in\C^2\hookrightarrow[1:z_1:z_2]\in{\mathbb P}^2$. Let 
\[Q(z_0,z_1)=\sum_{j=0}^da_jz_0^jz_1^{d-j}, 
\text{ where } a_0\neq0,\]
be a homogeneous polynomial of degree $d\geq2$. 
Set $P(z_1)=Q(1,z_1)$, so $P\in\C[z_1]$ is a polynomial of degree $d$. 
Let $X=X_Q$ be the algebraic curve of degree $d$ in ${\mathbb P}^2$ 
defined by 
\[X=\big\{[z_0:z_1:z_2]\in{\mathbb P}^2:\,z_0^{d-1}z_2-Q(z_0,z_1)=0\big\}.\]
Furthermore, let 
\[\omega=\omega_\FS\,\vert_X\,,\,\;L=\cO_{{\mathbb P}^2}(1)\,\vert_X,\]
where $\omega_\FS$ is the Fubini-Study form and 
$\cO_{{\mathbb P}^2}(1)$ is the hyperplane bundle on ${\mathbb P}^2$. 
Recall that $\pi^\star\omega_\FS=dd^c\log\|Z\|$, 
where $Z=(z_0,z_1,z_2)\in\C^3\setminus\{0\}$ and 
$\pi:\C^3\setminus\{0\}\to{\mathbb P}^2$ is the canonical projection. 
Moreover, if $U_j=\{z_j\neq0\}\subset{\mathbb P}^2$ 
then the transition functions of $\cO_{{\mathbb P}^2}(1)$ 
are given by $g_{jk}([z_0:z_1:z_2])=z_k/z_j$ on 
$U_j\cap U_k$, $0\leq j,k\leq2$ (see e.g.\ \cite[Example 4.4]{BCM20}). 

Note that in $\C^2\cong U_0$ we have 
$X\cap U_0=\{(z_1,z_2):\,z_2=P(z_1)\}$, 
so $X\cap U_0$ is biholomorphic to $\C$ via the obvious map 
$\zeta\in\C\to(\zeta,P(\zeta))\in X\cap U_0$. 
We also note that $X$ has one point on the line at infinity, 
as $X\cap\{z_0=0\}=\{[0:0:1]\}$. If $d\geq3$ then $X$ 
is singular and locally irreducible at $x_1:=[0:0:1]$, 
so $\Sigma=\{x_1\}$. It follows that the normalization of $X$ is 
${\mathbb P}^1$. In fact we have the explicit formula for this:
\[\sigma:{\mathbb P}^1\to{\mathbb P}^2\,,\,\;
\sigma([t_0:t_1])=[t_0^d:t_0^{d-1}t_1:Q(t_0,t_1)].\]
Here $[t_0:t_1]$ denote the homogeneous coordinates on 
${\mathbb P}^1$, and we and consider the standard 
embedding $\zeta\in\C\hookrightarrow[1:\zeta]\in{\mathbb P}^1$. 
Note that $\sigma([0:1])=x_1$ and for $r>0$ sufficiently small, the function 
\[f(t)=\sigma[t:1]=\left[\frac{t^d}{Q(t,1)}:
\frac{t^{d-1}}{Q(t,1)}:1\right],\;t\in\D_r,\]
is the local normalization of $X$ at $x_j$ as in \eqref{e:ln}. 
We infer that the ramification divisor of $\sigma$ (see \eqref{e:ramdiv}) is 
\[R_\sigma=(d-2)y, \text{ where } y=[0:1].\]

Since $X$ is locally irreducible at $x_1$ we have that 
$\cO_{X,c}(U)=\cO_{X,w}(U)$, and, by Lemma \ref{L:wsh3}, 
that $SH(U)=WSH(U)$, for any open set $U\subset X$. 
Let $S\in H^0_w(X,L^p)$ be represented by the holomorphic functions 
$s_0$ on $X\cap U_0$ and $s_2$ on $X\cap U_2$. 
Then there exists an entire function $s$ such that 
$s_0([1:\zeta:P(\zeta)])=s(\zeta)$ for $\zeta\in\C$, and we have 
\[s_0([1:\zeta:P(\zeta)])=P^p(\zeta)s_2([1:\zeta:P(\zeta)])\]
for all $\zeta$ with $|\zeta|$ sufficiently large. 
We infer that $s/P^p$ is bounded near $\infty$, 
hence $s$ is a polynomial of degree $\leq dp$. 
We conclude that the space $H^0_w(X,L^p)$ is isomorphic to 
$\C_{dp}[\zeta]$, hence $\dim H^0_w(X,L^p)=dp+1$. 
We can compare this to the subspace of restrictions to $X$ 
of global holomorphic sections on ${\mathbb P}^2$,
\[V_p:=\big\{S\,\vert_X:\,S\in H^0({\mathbb P}^2,L^p)\big\}.\]
If $H^0_X({\mathbb P}^2,L^p)=
\{S\in H^0({\mathbb P}^2,L^p):\,S=0 
\text{ on }X\}$ then $V_p\cong 
H^0({\mathbb P}^2,L^p)/H^0_X({\mathbb P}^2,L^p)$. 
Recall that sections in $H^0({\mathbb P}^2,L^p)$ 
are given by homogeneous polynomials of degree $p$ in 
$z_0,z_1,z_2$ so $\dim H^0({\mathbb P}^2,L^p)=\frac{(p+1)(p+2)}{2}$. 
Therefore 
\[\dim V_p=\frac{(p+1)(p+2)}{2}-
\frac{(p-d+1)(p-d+2)}{2}=dp-\frac{d(d-3)}{2}<dp+1=
\dim V_p,\]
since $d\geq3$. It is worth observing that 
$\sigma^\star:H^0_w(X,L^p)\to H^0({\mathbb P}^1,\sigma^\star L^p)$ 
is an isomorphism and $\sigma^\star L\cong\cO_{{\mathbb P}^1}(d)$.

We next describe the set of singular Hermitian metrics $h$ on $L$ 
that have subharmonic weights. Arguing as above, 
we infer that the weight of such $h$ on $X\cap U_0$ 
is given by a subharmonic function $\varphi$ on $\C$ 
such that $\varphi(\zeta)-\log|P(\zeta)|$ is bounded at infinity. Hence 
\begin{equation}\label{e:Lc}
\varphi(\zeta)\leq d\log^+|\zeta|+C_\varphi,\;\,\forall\,\zeta\in\C,
\end{equation}
with some constant $C_\varphi$. This shows that the set of 
singular Hermitian metrics on $L$ is in one-to-one correspondence 
to the class $d\mathcal L(\C)$, where $\mathcal L(\C)$ 
is the Lelong class of subharmonic functions of logarithmic growth on $\C$. 

We conclude this section by describing the Bergman space 
$H^0_{(2)}(X\setminus\Sigma,L^p)$ defined in \eqref{e:Bsreg}, 
where $h$ is the metric given by a function $\varphi\in d\mathcal L(\C)$. 
In view of the above, this space consists of the entire functions 
$s$ on $\C$ which verify 
\[\int_\C|s|^2e^{-2p\varphi}\,\sigma^\star(\omega)<+\infty.\]
We have 
\[\sigma^\star(\omega)=\frac{1}{2}\,
dd^c\log\big(1+|\zeta|^2+|P(\zeta)|^2\big)=
\frac{1+|P'(\zeta)|^2+|\zeta P'(\zeta)-
P(\zeta)|^2}{\pi\big(1+|\zeta|^2+
|P(\zeta)|^2\big)^2}\,\frac{i}{2}\,d\zeta\wedge d\ov\zeta.\]
Since $\zeta P'-P$ has degree $d$ we infer that 
\[\frac{c_1}{1+|\zeta|^{2d}}\,\frac{i}{2}\,d\zeta\wedge 
d\ov\zeta\leq\sigma^\star(\omega)\leq
\frac{c_2}{1+|\zeta|^{2d}}\,
\frac{i}{2}\,d\zeta\wedge d\ov\zeta\]
holds for $\zeta\in\C$, with some constants $c_1,c_2>0$. 
Using \eqref{e:Lc} we obtain that 
\[\int_{\C\setminus\D}|s(\zeta)|^2|\zeta|^{-2d(p+1)}\,d\lambda
<+\infty,\]
which implies that $s$ is a polynomial of degree 
$\leq dp+d-2$. In conclusion, $H^0_{(2)}(X\setminus\Sigma,L^p)$ 
is isomorphic to the space of polynomials $s\in\C_{dp+d-2}[\zeta]$ 
that verify 
\[\int_\C\frac{|s|^2e^{-2p\varphi}}{1+|\zeta|^{2d}}\,d\lambda
<+\infty,\]
where $\varphi\in d\mathcal L(\C)$ is the weight of $h$. 
We refer to the survey \cite{BCHM18} for results about 
equidistribution of zeros of random polynomials.
\end{Example}

\begin{Example}
Let $X$ be an irreducible algebraic curve and 
$(L,h)\to X$ be a Hermitian holomorphic line bundle
as in (A) and (B). We assume further that 
the Hermitian metric $h$ is smooth. 
Since $h$ is smooth, $H^0_{w,(2)}(X,L^p)=H^0_w(X,L^p)$
and $P_{w,p}(x)$ is defined for all $x\in X\setminus\Sigma$.
Let
$\sigma:\wi X\to X$ be the normalization of $X$,
$(\sigma^\ast L,\sigma^\ast h)\to\wi X$ the pull-back of $(L,h)$.
The curvature $c_1(\sigma^\ast L,\sigma^\ast h)$ is
semi-positive on $\wi X$, is positive on 
$\wi X\setminus\sigma^{-1}(\Sigma)$,
and vanishes up to finite order at any point of
$\sigma^{-1}(\Sigma)$. 
These are precisely the hypotheses of the results from
\cite{MS19} under which the Bergman kernel
asymptotics hold for a semi-positive line bundle
on a Riemann surface. 
There exists $C>0$ such that $\sigma^\ast\omega\leq C\wi{\omega}$
hence for any $S\in H^0_w(X,L^p)$ we have 
\[\|S\|^2_p=\int_{X\setminus\Sigma}|S|^2_{h_p}\,\omega
\leq C\int_{\wi{X}\setminus\sigma^{-1}
(\Sigma)}|\sigma^\ast S|^2_{\sigma^\ast h_p}\,\wi{\omega}=
C\int_{\wi{X}}|\sigma^\ast S|^2_{\sigma^\ast h_p}\,\wi{\omega}.\]
We consider the Bergman kernel function $\wi{P}_p$ of 
the space $H^0(\wi{X},\sigma^\ast L^p)$ with respect to the
Hermitian metric $\sigma^\ast h_p$ and volume form $\wi{\omega}$. 
By the isomorphism \eqref{e:iso} and the variational principle for the
Bergman kernel functions we have
\begin{equation}\label{e:com}
P_{w,p}(x)\geq\frac1C \wi{P}_p(\sigma^{-1}(x)),
\quad\text{for any $x\in X\setminus\Sigma$}\,.
\end{equation}
By \cite[Lemma 25]{MS19} there exists $\wi{C}>0$
such that for $p$ large enough we have 
\begin{equation}\label{e:lb0}
\wi{P}_p(\wi{x})\geq\wi{C}p^{2/r},
\quad\text{for any $\wi{x}\in\wi{X}$},
\end{equation}
where $r$ is the maximal normalized vanishing
order of the curvature $c_1(\sigma^\ast L,\sigma^\ast h)$
on $\wi{X}$, namely,
$r=\max\{r_{\wi{x}}:\wi{x}\in\wi{X}\}$,
$r_{\wi{x}}=\ord c_1(\sigma^\ast L,\sigma^\ast h)_{\wi{x}}+2$.
Hence, \eqref{e:com} and \eqref{e:lb0} show that 
there exists $C'>0$ such that for $p$ large enough we have
\begin{equation}\label{e:lb1}
P_{w,p}(x)\geq C'p^{2/r},
\quad\text{for any $x\in X\setminus\Sigma$}.
\end{equation}
Thus for $h$ smooth the convergence
$\frac1p\log P_{w,p}\to0$, as $p\to\infty$ in $L^1(X,\omega)$
follows directly from Lemma \ref{L:Bkf} and its proof. 
\end{Example}

\end{document}